\documentclass[11pt]{article}

\usepackage{amssymb,amsmath,amsfonts,amsthm}
\usepackage{latexsym}
\usepackage{graphics}
\usepackage{indentfirst}
\usepackage{hyperref}
\usepackage{comment}
\usepackage[shortlabels]{enumitem}
\usepackage[english]{babel}
\usepackage{tikz}
\usepackage{esdiff}

\usetikzlibrary{arrows.meta}

\newtheorem*{thm*}{Theorem}
\newtheorem{thm}{Theorem}
\newtheorem{lem}[thm]{Lemma}
\newtheorem{pro}[thm]{Proposition}

\newtheorem{cor}[thm]{Corollary}
\newtheorem{conj}[thm]{Conjecture}

\newtheorem{ques}[thm]{Question}

\newcommand{\N}{\mathbb{N}}

\allowdisplaybreaks

\begin{document}

\title{Flexible list colorings: Maximizing the number of requests satisfied 
}

\author{Hemanshu Kaul$^1$, Rogers Mathew$^2$, Jeffrey A. Mudrock$^3$, and Michael J. Pelsmajer$^4$}

\footnotetext[1]{Department of Applied Mathematics, Illinois Institute of Technology, Chicago, IL 60616.  E-mail:  {\tt {kaul@iit.edu}}}
\footnotetext[2]{Department of Computer Science and Engineering, Indian Institute of Technology, Hyderabad, India.  E-mail:  {\tt {rogers@cse.iith.ac.in}}}
\footnotetext[3]{Department of Mathematics and Statistics, University of South Alabama, Mobile, AL 36688.  E-mail:  {\tt {mudrock@southalabama.edu}}}
\footnotetext[4]{Department of Applied Mathematics, Illinois Institute of Technology, Chicago, IL 60616.  E-mail:  {\tt {pelsmajer@iit.edu}}}

\maketitle
\begin{abstract}
Flexible list coloring was introduced by Dvo\v{r}\'{a}k, Norin, and Postle in 2019.  Suppose $0 \leq \epsilon \leq 1$, $G$ is a graph, $L$ is a list assignment for $G$, and $r$ is a function with non-empty domain $D\subseteq V(G)$ such that $r(v) \in L(v)$ for each $v \in D$ ($r$ is called a request of $L$).  The triple $(G,L,r)$ is $\epsilon$-satisfiable if there exists a proper $L$-coloring $f$ of $G$ such that $f(v) = r(v)$ for at least $\epsilon|D|$ vertices in $D$.  We say $G$ is $(k, \epsilon)$-flexible if $(G,L',r')$ is $\epsilon$-satisfiable whenever $L'$ is a $k$-assignment for $G$ and $r'$ is a request of $L'$. It was shown by Dvo\v{r}\'{a}k et al. that if $d+1$ is prime, $G$ is a $d$-degenerate graph, and $r$ is a request for $G$ with domain of size~$1$, then $(G,L,r)$ is $1$-satisfiable whenever $L$ is a $(d+1)$-assignment. In this paper, we extend this result to all $d$ for bipartite $d$-degenerate graphs.

The literature on flexible list coloring tends to focus on showing that for a fixed graph $G$ and $k \in \N$ there exists an $\epsilon > 0$ such that $G$ is $(k, \epsilon)$-flexible, but it is natural to try to find the largest possible $\epsilon$ for which $G$ is $(k,\epsilon)$-flexible. In this vein, we improve a result of Dvo\v{r}\'{a}k et al., by showing $d$-degenerate graphs are $(d+2, 1/2^{d+1})$-flexible. 
In pursuit of the largest $\epsilon$ for which a graph is $(k,\epsilon)$-flexible,
we observe that  
a graph $G$ is not $(k, \epsilon)$-flexible for any $k$ if and only if $\epsilon > 1/ \rho(G)$, where $\rho(G)$ is the Hall ratio of $G$, and we initiate the study of the \emph{list flexibility number of a graph $G$}, which is the smallest $k$ such that $G$ is $(k,1/ \rho(G))$-flexible.  We study relationships and connections between the list flexibility number, list chromatic number, list packing number, and degeneracy of a graph.  

\medskip

\noindent {\bf Keywords.} list coloring, flexible list coloring, list coloring with requests, list packing, Hall ratio.

\noindent \textbf{Mathematics Subject Classification.} 05C15

\end{abstract}

\section{Introduction}\label{intro}

In this paper all graphs are nonempty, finite, undirected, and simple unless otherwise noted.  Generally speaking we follow West~\cite{W01} for terminology and notation.  The set of natural numbers is $\N = \{1,2,3, \ldots \}$.  For $m \in \N$, we write $[m]$ for the set $\{1, \ldots, m \}$.  If $u$ and $v$ are adjacent in a graph $G$, $uv$ or $vu$ refers to the edge between $u$ and $v$.  For $v \in V(G)$, we write $d_G(v)$ for the degree of vertex $v$ in the graph $G$ and $\Delta(G)$ for the maximum degree of a vertex in $G$.  We write $N_G(v)$ (resp., $N_G[v]$) for the neighborhood (resp., closed neighborhood) of vertex $v$ in the graph $G$.  If $S \subseteq V(G)$, we use $G[S]$ for the subgraph of $G$ induced by $S$.  We write $H \subseteq G$ when $H$ is a subgraph of $G$.  The chromatic number of $G$ is denoted $\chi(G)$, and the independence number of $G$ is denoted $\alpha(G)$.  A graph is $d$-degenerate if there exists an ordering of the vertices of the graph such that each vertex in the ordering has at most $d$ neighbors preceding it in the ordering.  The Cartesian product of graphs $G$ and $H$, denoted $G \square H$, is the graph with vertex set $V(G) \times V(H)$ and edges created so that $(u,v)$ is adjacent to $(u',v')$ if and only if either $u=u'$ and $vv' \in E(H)$ or $v=v'$ and $uu' \in E(G)$.  The join of vertex disjoint graphs $G$ and $H$, denoted $G \vee H$, is the union of $G$ and $H$ and all edges $uv$ where $u$ is in $G$ and $v$ is in $H$. The square of a graph $G$, denoted $G^2$, is the graph on $V(G)$ where vertices $u,v$ are adjacent in $G^2$ if their distance $d(u,v)$ is at most~2 in $G$: for each edge $uv$ of $G^2$, either $uv$ is an edge of $G$ or $u$ and $v$ have a common neighbor in $G$.

\subsection{Flexible list colorings} 

In classical vertex coloring we wish to color the vertices of a graph $G$ with up to $m$ colors from $[m]$ so that adjacent vertices receive different colors, a so-called \emph{proper $m$-coloring}.  List coloring is a well-known variation on classical vertex coloring that was introduced independently by Vizing~\cite{V76} and Erd\H{o}s, Rubin, and Taylor~\cite{ET79} in the 1970s.  For list coloring, we associate a \emph{list assignment} $L$ with a graph $G$ such that each vertex $v \in V(G)$ is assigned a list of colors $L(v)$ (we say $L$ is a list assignment for $G$).  An \emph{$L$-coloring} of $G$ is a function $f$ with domain $V(G)$ such that $f(v)\in L(v)$ for all $v\in V(G)$.  We say that $G$ is $L$-colorable if there exists a proper $L$-coloring of $G$: an $L$-coloring where adjacent vertices receive different colors.  A list assignment $L$ for $G$ is called a \emph{$k$-assignment} if $|L(v)|=k$ for each $v \in V(G)$.  We say $G$ is \emph{$k$-choosable} or \emph{$k$-list colorable} if $G$ is $L$-colorable whenever $L$ is a $k$-assignment for $G$.  The \emph{list chromatic number} of $G$, denoted $\chi_{\ell}(G)$, is the smallest $k$ such that $G$ is $k$-choosable.  It is immediately obvious that for any graph $G$, $\chi(G) \leq \chi_{\ell}(G)$.

Flexible list coloring was introduced by Dvo\v{r}\'{a}k, Norin, and Postle in~\cite{DN19} in order to address a situation in list coloring where we still seek a proper list coloring, but each vertex may have a preferred color assigned to it, and for those vertices we wish to color as many of them with their preferred colors as possible.  (This is ``flexible'' as compared to the classical precoloring extension problem, where some vertices are assigned colors that they must use.)  Specifically, suppose $G$ is a graph and $L$ is a list assignment for $G$.  A \emph{request} of $L$ is a function $r$ with non-empty domain $D\subseteq V(G)$ such that $r(v) \in L(v)$ for each $v \in D$.  For any $\epsilon \in (0,1]$, the triple $(G,L,r)$ is \emph{$\epsilon$-satisfiable} if there exists a proper $L$-coloring $f$ of $G$ such that $f(v) = r(v)$ for at least $\epsilon|D|$ vertices in $D$.  

We say that the pair $(G,L)$ is \emph{$\epsilon$-flexible} if $(G,L,r)$ is $\epsilon$-satisfiable whenever $r$ is a request of $L$.  Finally, we say that $G$ is \emph{$(k, \epsilon)$-flexible} if $(G,L)$ is $\epsilon$-flexible whenever $L$ is a $k$-assignment for $G$.  If $G$ is $(k,\epsilon)$-flexible, the following observations are immediate: (i) $G$ is $(k',\epsilon')$-flexible for any $k'\ge k$ and $\epsilon'\le \epsilon$; (ii) any spanning subgraph $H$ of $G$ is $(k,\epsilon)$-flexible; (iii) $G$ is $k$-choosable.

Dvo\v{r}\'{a}k, Norin, and Postle considered  $d$-degenerate graphs, which are known to be $(d+1)$-choosable.  They showed that $d$-degenerate graphs are all $(d+2,\epsilon)$-flexible for some $\epsilon>0$.  
More specifically, their proof (see Lemma~11 in~\cite{DN19}) gives the following result.

\begin{thm}[\cite{DN19}]
\label{thm:Dvorak_et_al_degeneracy}
For any integer $d\ge 0$, every $d$-degenerate graph is $\displaystyle \left({d+2, \frac{1}{(d+2)^{(d+1)^2}}}\right)$-flexible. 
\end{thm}

Dvo\v{r}\'{a}k, Norin, and Postle also gave special consideration to requests with domain of size 1 as such requests can often be the main obstacle to $(k,\epsilon)$-flexibility.  In particular, they proved the following using the Combinatorial Nullstellensatz.   

\begin{thm}[\cite{DN19}]
\label{thm:Dvorak_et_al_nullstellensatz}
Let $d\ge 2$ such that $d+1$ is a prime. If $G$ is a $d$-degenerate graph and $r$ is a request for $G$ with domain of size~$1$, then $(G,L,r)$ is $1$-satisfiable whenever $L$ is a $(d+1)$-assignment.  
\end{thm}



Flexible list coloring has been studied extensively for planar graphs that are $5$-choosable, and for restricted subclasses of planar graphs that are $k$-choosable with $k<5$.  It has been shown that there exists an $\epsilon > 0$ such that every planar graph $G$ is $(6,\epsilon)$-flexible~\cite{DN19}. Several papers~\cite{choi_planar_sharp_tool_flexibility,dvorak_planar_girth_flexibility,dvovrak_triangle_free_planar_flexibility,DN19, lidicky_weak_flexibility,planar_without_4cycle_flexibility,yang_c4c6_free_planar_flexibility, yang_sufficient_conditions_planar_flexible} have explored $(k,\epsilon)$-flexibility, with $k\in \{5,4,3\}$, of planar graphs with large enough girth or excluding certain cycles. A `weak' notion of $\epsilon$-flexibility where the domain of the request set is always the entire vertex set is studied in the context of planar graphs in~\cite{choi_planar_sharp_tool_flexibility, lidicky_weak_flexibility}. 

\subsection{Satisfying more of the request}

The literature on flexible list coloring tends to focus on showing that for a fixed graph $G$ and $k \in \N$ there exists an $\epsilon > 0$ such that $G$ is $(k, \epsilon)$-flexible. But it is natural to try to find the largest possible $\epsilon$ for which $G$ is $(k,\epsilon)$-flexible; that is, one would prefer to have a larger portion of the requested colors on vertices satisfied. The authors are only aware of one paper that contains results related to improving the value of $\epsilon$ for fixed $k$ (see~\cite{BM22}).  For example, one result proven in~\cite{BM22} is as follows.  

\begin{thm} [\cite{BM22}] \label{thm: flexbrooks}
Let $G$ be a connected graph with $\Delta(G) \geq 3$ that is not a copy of $K_{\Delta(G)+1}$.  Then, $G$ is $(\Delta(G), 1/(6\Delta(G)))$-flexible.  Moreover, $1/(6\Delta(G))$ is within a constant factor of being best possible.
\end{thm}

In Section~\ref{d-degenerate}, we improve Theorem~\ref{thm:Dvorak_et_al_degeneracy} by giving a larger $\epsilon$ for $d$-degenerate graphs with a simpler proof.

\begin{thm} \label{thm: seemseasy}
Suppose $G$ is $d$-degenerate.  Then, $G$ is  $(d+2, \frac{1}{2^{d+1}})$-flexible.
\end{thm}

Graphs that are $(k,\epsilon)$-flexible for some $\epsilon>0$ and graphs that are $d$-degenerate for $k=d+1$ are both $k$-choosable.  Consequently, in Section~\ref{d-degenerate}, we also consider flexible list coloring for $k$-choosable graphs more generally.

\begin{pro}
\label{prop:chromatic_no_of_square_of_G}
Let $G$ be an $s$-choosable graph. Then, $G$ is $(s+1, 1/\chi(G^2))$-flexible. 
\end{pro}

If $G$ is a $d$-degenerate graph with maximum degree $\Delta$, then $G^2$ is $\left (\Delta(2d-1) + d - d^2 \right )$-degenerate; we show this in Section~\ref{d-degenerate}.
This yields the following corollary to Proposition~\ref{prop:chromatic_no_of_square_of_G}, which is an improvement on Theorem~\ref{thm: seemseasy} for $d$-degenerate graphs with maximum degree $\Delta< 2^d/d$.

\begin{cor} \label{cor:chromatic_no_of_square_of_G}
Let $G$ be a $d$-degenerate graph with maximum degree $\Delta$. If $G$ is $s$-choosable, then 
$G$ is $(s+1, 1/(\Delta(2d-1) + d - d^2 + 1))$-flexible.   
\end{cor} 

We can obtain larger $\epsilon$ for $s$-choosable graphs in the case where the total number of requested colors is bounded, albeit for larger list sizes.

\begin{pro} \label{pro: useful?}
 Let $G$ be $s$-choosable. Suppose $L$ is an $(s+k)$-assignment for $G$. Let $r$ be a request of $L$ with domain $D$. Suppose the requested palette, $r(D)$, satisfies $|r(D)| \leq k$. Then, $(G,L,r)$ is $(1/\chi(G))$-satisfiable. 
\end{pro}

This proposition is also proved in Section~\ref{d-degenerate}.

\subsection{Maximizing the number of vertex requests satisfied}

For each graph $G$, what is the largest $\epsilon$ so that $G$ is $(k,\epsilon)$-flexible for some $k$?  It is possible that $r(v)$ is the same color for all $v\in D$; for example, let $L$ be the $k$-list assignment such that $L(v)=[k]$ for all $v\in V(G)$, and let $r(v)=1$ for all $v\in D$.  Then at most $\alpha(G[D])$ vertices in $D$ will have their request fulfilled.  Hence $\epsilon |D|\le \alpha(G[D])$ for every non-empty $D\subseteq V(G)$ when $G$ is $(k,\epsilon)$-flexible, regardless of $k$.  That is, $\epsilon\le \min_{\emptyset\not=D\subseteq V(G)} \alpha(G[D])/|D|$ for any $(k,\epsilon)$-flexible graph $G$.

The \emph{Hall ratio} of a graph $G$ is 
$$\rho(G) = \max_{\emptyset\not= H \subseteq G} \frac{|V(H)|}{\alpha(H)}.$$  
The Hall ratio was first studied in 1990 by Hilton and Johnson Jr.~\cite{HJ90} under the name Hall-condition number in the context of list coloring. In the past 30 years, the Hall ratio has received much attention due to its connection with both list and fractional coloring (e.g. see~\cite{BL22, CC21, CJ00, DM20}).

Note that among subgraphs $H$ of $G$ with fixed vertex set $D$, $\alpha(H)$ is minimized by the induced subgraph $H=G[D]$.  Therefore, $\min_{\emptyset\not=D\subseteq V(G)} \alpha(G[D])/|D| = 1/\rho(G)$.  Moreover, this bound on feasible $\epsilon$ for $(k,\epsilon)$-flexibility is attainable.

\begin{pro} \label{pro: fundamental}
There exists $k$ such that $G$ is $(k,\epsilon)$-flexible if and only if $\epsilon\le 1/\rho(G)$.
\end{pro}

\begin{proof}
We have already observed that $G$ is not $(k,\epsilon)$-flexible when $\epsilon> 1/\rho(G)$.  We will show that any graph $G$ is $(k,\epsilon)$-flexible for $k\ge \Delta(G)+1$ and $\epsilon\le 1/\rho(G)$.
Let $L$ be any $k$-assignment for $G$ with $k\ge \Delta(G)+1$, and let $r$ be a request of $L$ with any non-empty domain $D\subseteq V(G)$.  Let $I$ be an independent set in $G[D]$ of size $\alpha(G[D])$.  We can obtain a proper $L$-coloring of $G$ by first coloring each $v \in I$ with $r(v)$, then we can greedily color the remaining vertices since $k\ge \Delta(G)+1$.  Since each vertex $v\in I$ gets color $r(v)$ and $|I|=\alpha(G[D])\ge |D|/\rho(G)$, $(G,L,r)$ is $\epsilon$-satisfiable for $\epsilon\le 1/\rho(G)$.
\end{proof}

Since we would like to maximize $\epsilon$ for which
$G$ is $(k,\epsilon)$-flexible, we define the \emph{list flexibility number} of $G$, denoted $\chi_{flex}(G)$, to be the smallest $k$ such that $G$ is $(k,1/\rho(G))$-flexible.  

It was shown in~\cite{BM22} that $\chi_{flex}(T) = 2$ whenever $T$ is a tree with at least one edge. From the proof of Proposition~\ref{pro: fundamental}, we have $\chi_{flex}(G) \leq \Delta(G)+1$ for any graph $G$.  Also, $\chi_{\ell}(G) \leq \chi_{flex}(G)$. These two facts imply that $\chi_{flex}(K_n) = n$ and $\chi_{flex}(C_k)=3$ for odd $k$. It is natural to ask whether a Brooks-type theorem is true for $\chi_{flex}$ as well.

\begin{ques} \label{ques: Brooks}
What are all the graphs $G$ such that $\chi_{flex}(G) = \Delta(G)+1$?   
\end{ques}

Dvo\v{r}\'{a}k, Norin, and Postle conjectured that $d$-degenerate graphs $G$ are $(d+1,\epsilon)$-flexible for some $\epsilon>0$, which would strengthen both their Theorem~\ref{thm:Dvorak_et_al_degeneracy} and the easy list coloring bound of $\chi_\ell(G)\le d+1$.  One might try to disprove the even stronger analogous statement for $\epsilon=1/\rho(G)$.

\begin{ques} \label{ques: degen}
Does there exist a graph $G$ with degeneracy $d$ satisfying $\chi_{flex}(G) > d+1$?   
\end{ques}

Recall that the list chromatic number is a lower bound on the list flexibility number.  A famous result of Alon~\cite{A00} says that for any graph $G$ with degeneracy $d$, $(1/2 - o(1) ) \log_2(d+1) \leq \chi_\ell(G)$
; this bound is tight up to a factor $2+o(1)$. Consequently, $\chi_{flex}(G)$ can be bounded below by a function of the degeneracy of $G$ when the degeneracy of $G$ is sufficiently large. So it is natural to ask whether this lower bound is also tight.

\begin{ques} \label{ques: lowerbounddegen}
Suppose $\mathcal{F}(d) = \min \{\chi_{flex}(G): \text{the degeneracy of $G$ is at least $d$} \}$.  What is the asymptotic behavior of $\mathcal{F}(d)$ as $d \rightarrow \infty$? 
\end{ques}

Note that by Alon's result we know that $\mathcal{F}(d) = \Omega(\log_2(d))$ as $d \rightarrow \infty$. In Section~\ref{joins}, we study the list flexibility number of join of graphs with the aim of better understanding the list flexibility number with respect to degeneracy. Theorem~\ref{thm: generaljoin} is a general upper bound on $\chi_{flex}(G\vee G)$ and Corollary~\ref{cor: paths} is a better bound for the join of paths: $\chi_{flex}(P_n\vee P_n)\le \lceil{n/2+\ln n}\rceil$ for $n\ge 50$.  Since $P_n\vee P_n$ is $(n+1)$-degenerate, $\mathcal{F}(d) \le (1+o(1))d/2$ as $d \rightarrow \infty$. 

It is also natural to ask whether $\chi_{flex}(G)$ can be bounded above by a function of $\chi_\ell(G)$.  

\begin{ques} \label{ques: listchromatic}
Does there exist a function $f$ such that for every graph $G$, $\chi_{flex}(G)  \leq f(\chi_{\ell}(G))$? 
\end{ques}

In Section~\ref{versus} (Proposition~\ref{pro: oddrequest}), we show that there is no universal constant $c$ such that $\chi_{flex}(G) \le \chi_{\ell}(G) +c$ by showing that there are graphs $G$ for which 
$\chi_{flex}(G) > {4\over 3}\chi_{\ell}(G)$. See also Conjectures~\ref{conj: pack} and~\ref{conj: coulditbe} in the next subsection.

One might question whether maximizing $\epsilon$ is meaningfully different from a flexible list coloring with smaller $\epsilon>0$.  Are there any graphs $G$ and $k \in \N$ such that $G$ is $(k,\epsilon)$-flexible for some $\epsilon>0$, but $G$ is not $(k,1/\rho(G))$-flexible? The following result shows that the answer is yes.

\begin{pro} \label{pro: K37}
Suppose $G = K_{3,7}$.  Then, $G$ is $(3, 1/10)$-flexible and $\chi_{flex}(G) > 3$.
\end{pro}

We prove this result in Section~\ref{versus}.

\medskip

While it is not directly pursued in this paper, one could also investigate for each graph $G$, the maximum possible $\epsilon$ for fixed $k$. Let $\epsilon_G$ be the function that maps each $k\in \N$ to the largest $\epsilon$ such that $G$ is $(k,\epsilon)$-flexible.  Clearly, $\epsilon_G(k)=a/b$ for some integers $0\le a\le b\le |V(G)|$ and $\epsilon_G(k)\le 1/\rho(G)$. It would be interesting to study $\epsilon_G$ for various $G$, so we propose this as a subject of future research.

\subsection{List packing} \label{subsec: list packing}

List packing is a relatively new notion that was first suggested by Alon, Fellows, and Hare in 1996~\cite{AF96}.  This suggestion was not formally embraced until a recent paper of Cambie, van Batenburg, Davies, and Kang~\cite{CC21}.  We now mention some important definitions.  Suppose $L$ is a list assignment for a graph $G$.  An \emph{$L$-packing of $G$ of size $k$} is a set of $L$-colorings 
$\{f_1, \ldots, f_k \}$ of $G$ such that $f_i(v) \neq f_j(v)$ whenever $i, j \in [k]$, $i \neq j$, and $v \in V(G)$.  Moreover, we say that $\{f_1, \ldots, f_k \}$ is \emph{proper} if $f_i$ is a proper $L$-coloring of $G$ for each $i \in [k]$. The \emph{list packing number} of $G$, denoted $\chi_{\ell}^*(G)$, is the least $k$ such that $G$ has a proper $L$-packing of size $k$ whenever $L$ is a $k$-assignment for $G$.  Clearly, for any graph $G$, $\chi(G) \leq \chi_{\ell}(G) \leq \chi_{\ell}^*(G)$, and it is well-defined~\cite{CC21,M22}.

We study and formalize a connection between
list packing
and $(k, \epsilon)$-flexibility that was already inherent in~\cite{BM22}.

\begin{pro} \label{pro: pack}
For any graph $G$, $G$ is $(\chi_{\ell}^*(G) , 1/\chi_{\ell}^*(G))$-flexible.  
\end{pro}

Our results are motivated by questions that focus on bounding $\chi_{flex}(G)$ for arbitrary $G$.  In particular, with Proposition~\ref{pro: pack} in mind, we conjecture the following.

\begin{conj}
\label{conj: pack}
\footnote{After an earlier version of this paper was posted, Conjecture \ref{conj: pack} was disproved by Cambie and H{\"a}m{\"a}l{\"a}inen~\cite{cambie2023packing} with the help of Proposition~\ref{pro: K37}. The question of infinitely many counterexamples remains open.}

For any graph $G$, $\chi_{flex}(G) \leq \chi_{\ell}^*(G)$.
\end{conj}

The following tantalizing conjecture is in~\cite{CC21}.

\begin{conj} [\cite{CC21}] \label{conj: coulditbe}
There exists a $C > 0$ such that $\chi_{\ell}^*(G) \leq C \cdot \chi_{\ell}(G)$ for any graph $G$.
\end{conj}

It can be immediately noted that if Conjectures~\ref{conj: pack} and~\ref{conj: coulditbe} are both true, then the answer to Question~\ref{ques: listchromatic} would be yes in a very strong sense.

As partial evidence towards Conjecture~\ref{conj: pack}, it follows from Proposition~\ref{pro: pack} and a result of~\cite{CC21} bounding $\chi_{\ell}^*(G)$ in terms of $\rho(G)$, that every graph $G$ on $n$ vertices is $(\chi_{\ell}^*(G), 1/ ((5 +o(1)) \rho(G) (\log n)^2))$-flexible where the $o(1)$ term tends to 0 as $n$ tends to infinity.


Proposition~\ref{pro: pack} follows immediately from a more general result, Proposition~\ref{pro: generalized pack},\footnote{In terms of the {\it fractional list packing number} $\chi_\ell^\bullet$, which is defined in~[6], Proposition~\ref{pro: generalized pack} yields the stronger result that $G$ is $(\chi_{\ell}^\bullet(G) , 1/\chi_{\ell}^\bullet(G))$-flexible.} which is implicit in~\cite{DN19} and used frequently in~\cite{BM22}. 

\begin{pro} \label{pro: generalized pack}
Suppose that $G$ is a graph, $L$ is a $k$-assignment for $G$, and there is a set $\mathcal S$ of $mk$ proper $L$-colorings such that for each vertex $v\in V(G)$ and each color $c\in L(v)$, vertex $v$ is colored by $c$ in exactly $m$ of the $L$-colorings of $\mathcal S$.  Then, $(G,L,r)$ is $1/k$-satisfiable for any request $r$ of $L$.
\end{pro}

\begin{proof}
Fix a request $r$ of $L$ and let $D$ be its domain.  
For each $f\in {\mathcal S}$ and $v\in V(G)$, let $i(f,v)$ be the indicator function that is 1 if $f(v)=r(v)$ and 0 if $f(v)\not=r(v)$.  The sum of $i$ over all $f\in {\mathcal S}$ and $v\in D$ equals $m|D|$, so by the Pigeonhole Principle there must exist a coloring $f\in {\mathcal S}$ such that at least $m|D|/(mk)$ vertices of $D$ satisfy $f(v)=r(v)$; that is, $(G,L,r)$ is $1/k$-satisfiable.
\end{proof}

For instance, suppose that $T$ is a nontrivial tree.  Theorem~3 in~\cite{CC21} implies that $\chi_{\ell}^*(T)=2$. Since $\rho(T)=2$ and $\chi_{flex}(T)\ge \chi_{\ell}(T) = 2$, Proposition~\ref{pro: pack} implies $\chi_{flex}(T)=2$.  We will use $\chi_{\ell}^*(P_n) = 2$ and Proposition~\ref{pro: generalized pack} to obtain the following result.\footnote{In terms of fractional list packing number, the proof of Proposition~\ref{pro: grid} shows that $\chi_\ell^\bullet(P_n\square P_m) = 3$ for
$n,m\ge 2$, with which Proposition~\ref{pro: generalized pack} implies Proposition~\ref{pro: grid}.}

\begin{pro} \label{pro: grid}
The grid $P_n \square P_m$ is $(3, 1/3)$-flexible.
\end{pro}

Since $G=P_m\square P_n$ with $m,n\ge 2$ is $2$-colorable, $\rho(G)=2$.  Thus, $\chi_{flex}(G)$ is the minimum $k$ such that $G$ is $(k,1/2)$-flexible.  Since it is known $\chi_{flex}(C_4) \geq 3$ (see~\cite{BM22}), and $G$ contains a $C_4$, $\chi_{flex}(G) \geq 3$.  Using a detailed inductive proof, we are able to obtain the best possible result for the $n$-ladder $P_2\square P_n$.

\begin{pro} \label{pro: ladder}
Suppose $G = P_2 \square P_n$ with $n \geq 2$.  Then, $G$ is $(3,1/2)$-flexible.  Consequently, $\chi_{flex}(G)=3$. 
\end{pro}

Propositions~\ref{pro: grid} and~\ref{pro: ladder} are both proved in
Section~\ref{products}.  We also obtain results for graph products in general (Proposition~\ref{pro: Cartesian} in Section~\ref{products}).  

\subsection{Small requests}

On the other hand, if we focus solely on $k$ and allow arbitrarily small $\epsilon>0$, then our colorings need only satisfy the color request at a single  vertex.  This is hardest to satisfy when a request has domain of size~1, so without loss of generality we only consider such requests.

For a graph $G$ with list assignment $L$ and request $r$ with domain $D$ of size~1, $(G,L,r)$ is either 1-satisfiable or $(G,L,r)$ is not $\epsilon$-satisfiable for any $\epsilon > 0$.  Thus, without loss of generality, $\epsilon=1$.
As mentioned earlier, these considerations ultimately led Dvo\v{r}\'{a}k, Norin, and Postle in~\cite{DN19} to pay special attention to this case and prove Theorem~\ref{thm:Dvorak_et_al_nullstellensatz}. 
Lemma~\ref{lem: K37} is another result of this form. 
Proposition~\ref{pro: useful?} considers another sort of trade-off between $k$ and~$\epsilon$.

Using the Alon-Tarsi Theorem~\cite{alon1992colorings}, we are able to extend Theorem~\ref{thm:Dvorak_et_al_nullstellensatz} to all $d$ for bipartite $d$-degenerate graphs.

\begin{thm}\label{alontarsi}
For any bipartite $d$-degenerate graph $G$ with a $(d+1)$-assignment $L$ and request $r$ with domain $D$ of size~1, $(G,L,r)$ is 1-satisfiable. 
\end{thm}

\subsubsection*{Organization of the rest of the paper}
In Section~\ref{d-degenerate}, we give proofs of Theorem \ref{thm: seemseasy}, Proposition \ref{prop:chromatic_no_of_square_of_G}, Corollary \ref{cor:chromatic_no_of_square_of_G}, and Proposition \ref{pro: useful?}. 
Theorem~\ref{alontarsi} is proved in Section~\ref{sec: alontarsi}. 
We prove Proposition \ref{pro: K37}  in Section \ref{versus}. In the same section, we prove Proposition \ref{pro: oddrequest} which shows that there are graphs $G$ for which $\chi_{flex}(G) > {4\over 3}\chi_{\ell}(G)$. 
In Section \ref{products}, we prove Propositions \ref{pro: grid} and \ref{pro: ladder} relating to Cartesian products of graphs. Finally, in Section \ref{joins} we prove results relating to the list flexibility number of the join of two copies of the same graph. 

\section{Flexible list coloring for $d$-degenerate and $s$-choosable graphs}\label{d-degenerate}

We begin by proving Theorem~\ref{thm: seemseasy}, that every $d$-degenerate graph is $(d+2, \frac{1}{2^{d+1}})$-flexible.

\begin{proof}
Let $v_1, \ldots , v_n$ be an ordering of the elements of $V(G)$ such that every $v_i$ has at most $d$ neighbors preceding it in the ordering. Let $L$ be a $(d+2)$-assignment for $G$, and let $r~:~D \rightarrow \bigcup_{v_i \in D}L(v_i)$ be a request of $L$. For each $v_i \in V(G)$, let $\sigma_i$ be an ordering of the colors in $L(v_i)$ such that if $v_i \in D$ then $r(v_i)$ is the first color in this ordering. We construct a proper $L$-coloring, $f$, of $G$ using a random process. Our process begins with $v_1$ and works its way to $v_n$ by following the ordering $v_1, \ldots, v_n$. Assume we have colored $v_1, \ldots , v_{i-1}$ and we wish to color $v_i$ for some $i \in [n]$. The list $L(v_i)$ has at least two `unused colors'  available that have not been used on any of the neighbors of $v_i$ preceding it in the ordering (since $|L(v_i)| = d+2$ and $v_i$ has at most $d$ neighbors preceding it). Based on the ordering $\sigma_i$ of the colors in $L(v_i)$, let $S_i$ be the first two unused colors in $L(v_i)$. Let $f(v_i)$ be a color chosen uniformly at random from $S_i$. We proceed in this manner to color all the vertices of $G$. It is easy to see that $f$ is a proper $L$-coloring of $G$. For any $v_i \in V(G)$, and a color $c \in \bigcup_{j=1}^{n} L(v_j)$, we define the indicator random variable $X_i^c$ as 
\[X_i^c=\begin{cases}
	1 \text{, if }f(v_i) = c;\\
	0 \text{, otherwise.}
	\end{cases} \]    
Suppose $v_i \in D$ and $c_i := r(v_i)$. Let $A_i$ be the event that $c_i$ is available as an unused color at $v_i$. Let $B_i$ be the event that $f(v_i) = c_i$. Note that if $c_i$ is indeed available as an unused color at $v_i$, the way we construct $S_i$ ensures that $c_i$ is in $S_i$. We have
\begin{eqnarray}
\mathbb{P}[X_i^{c_i} = 1] & = & \mathbb{P}[A_i \cap B_i] \nonumber \\
 & = & \mathbb{P}[A_i] \cdot \mathbb{P}[B_i | A_i]  \nonumber \\
 & = & \mathbb{P}[A_i] \cdot \frac{1}{2}. \label{ineq:prob_1}
\end{eqnarray}
Let $v_{i_1}, v_{i_2}, \ldots , v_{i_k}$ with $i_1 < i_2 < \cdots < i_k < i$ be the collection of all the neighbors of $v_i$ among $v_1, \ldots, v_{i-1}$. Clearly, $k \leq d$. So, 
\begin{eqnarray}
\mathbb{P}[A_i] & = & \mathbb{P}[(X_{i_1}^{c_i} = 0) \wedge (X_{i_2}^{c_i} = 0) \wedge \cdots \wedge (X_{i_k}^{c_i} = 0)] \nonumber \\
& = & \mathbb{P}[(X_{i_1}^{c_i} = 0)] \cdot \mathbb{P}[(X_{i_2}^{c_i} = 0) | (X_{i_1}^{c_i} = 0)] \cdots \nonumber \\ 
& & \mathbb{P}[(X_{i_k}^{c_i} = 0) | (X_{i_1}^{c_i} = 0) \wedge \cdots  \wedge (X_{i_{k-1}}^{c_i} = 0)]   \nonumber \\
& \geq & \frac{1}{2^k} \nonumber \\
& \geq & \frac{1}{2^d}. \label{ineq:prob_2}
\end{eqnarray}
Combining (\ref{ineq:prob_1}) and (\ref{ineq:prob_2}), we have $\mathbb{P}[X_i^{c_i} = 1] \geq \frac{1}{2^{d+1}}$. Thus  for every $v_j \in D$ with $c_j := r(v_j)$, we have $\mathbb{E}[X_j^{c_j}] \geq \frac{1}{2^{d+1}}$. Let $X = \sum_{v_j \in D} X_j^{c_j}$, and note that $X$ is the random variable that  denotes the number of requests satisfied by $f$. By linearity of expectation, we get $\mathbb{E}[X] \geq \frac{1}{2^{d+1}}|D|$. 
\end{proof}

An $(s,\epsilon)$-flexible graph must be $s$-choosable.  Without $d$-degeneracy, we can get results for $s$-choosable graphs as long as the palette of requested colors $r(D)=\{r(v)\,:\,v\in D\}$ is not too large, say at most $k$.  In particular, Proposition~\ref{pro: useful?}, which we now prove, shows that for any $s$-choosable graph $G$, if $r$ is a request $r$ satisfying $|r(D)|\le k$ for an $(s+k)$-list assignment $L$, then  $(G,L,r)$ is $(1/\chi(G))$-satisfiable.

\begin{proof}
We will first color some vertices $v\in D$ with color $r(v)$, then extend the coloring to $G$.

Let $R_P = r(D)$.  Fix a $\chi(G)$-coloring of $G$. 
Independently and uniformly at random, assign each color in $R_P$ to any one of the $\chi(G)$ color classes. 
For each vertex $v \in D$, we color $v$ with $r(v)$ if and only if $r(v)$ was assigned to the color class containing $v$.  Thus, $v$ gets color $r(v)$ with probability $1/\chi(G)$. The expected number of vertices in $D$ that receive the requested color is $|D|/\chi(G)$, so there is a way to $L$-color at least that many vertices $v$ with their requested colors.

Next we extend the $L$-coloring to the vertices of $V(G) - D$ and the uncolored vertices of $D$.  
For each uncolored vertex $w$, let $L'(w) = L(w) - R_P$. Since $|R_P| \leq k$, we have $|L'(w)| \geq s$. Since $G$ is $s$-choosable, we can clearly complete a proper $L$-coloring of $G$. 
\end{proof}


Without restricting the degeneracy or the request, we can obtain a general bound on $\epsilon$ in terms of the square of the graph.  Recall that for a graph $G$, its square $G^2$ is the graph with $V(G^2) = V(G)$ and $E(G^2) = \{uv~:~u,v \in V(G),~dist_G(u,v) \leq 2\}$. 

We now prove Proposition~\ref{prop:chromatic_no_of_square_of_G}, that every $s$-choosable graph is $(s + 1, 1/\chi(G^2))$-flexible.

\begin{proof}
Let $L$ be an $(s+1)$-assignment for $G$. Let $r:~D \rightarrow \bigcup_{v \in D}L(v)$ be a request of $L$, where $D \subseteq V(G)$ is the domain of the request. Consider a proper coloring of $G^2$ using $\chi(G^2)$ colors. By pigeonhole principle, there exists a color class $C$ in this coloring such that $|C \cap D| \geq \frac{|D|}{\chi(G^2)}$. Moreover, no two vertices in $C$ have a common neighbor in $G$. Let $C_D = C \cap D$. We now construct a proper $L$-coloring $\phi$ of $G$ that satisfies the requests of all the vertices in $C_D$. For every $v \in C_D$, let $\phi_0(v) = r(v)$. Consider a vertex $u \in V(G) \setminus C_D$. If $N_G(u) \cap C_D = \emptyset$, then let $L'(u) = L(u)$. Otherwise, let $L'(u) = L(u) - \{r(v)~:~v \in N_G(u) \cap C_D\}$. Since no two vertices in $C$ can have a common neighbor in $G$, we have $|L'(u)| \geq s$, for every $u \in V(G) - C_D$. Let $\phi_1$ be a proper $L'$-coloring of $G[V - C_D]$. Such a coloring exists as $G$ (and thereby all its subgraphs) are $s$-choosable. We now have the desired $L$-coloring $\phi$ of $G$ which is defined as 
\[\phi(v)=\begin{cases}
	\phi_0(v) \text{, if }v \in C_D;\\
	\phi_1(v) \text{, otherwise.}
	\end{cases} \]  
This completes the proof. 	  
\end{proof}

Finally, we show how Proposition~\ref{prop:chromatic_no_of_square_of_G} implies Corollary~\ref{cor:chromatic_no_of_square_of_G}.

\begin{proof}
Let $G$ be a $d$-degenerate graph with maximum degree $\Delta$.  We will show that $G^2$ is $\left(\Delta(2d-1) + d - d^2\right)$-degenerate.  Then $\chi(G^2)\le \Delta(2d-1) + d - d^2 + 1$, with which Proposition~\ref{prop:chromatic_no_of_square_of_G} implies Corollary~\ref{cor:chromatic_no_of_square_of_G}.

Since $G$ is $d$-degenerate, $G$ has a vertex ordering so that each vertex $u$ has at most $d$ neighbors $v$ preceding it; we will use the same ordering for $G^2$.  Each such $v$ has at most $\Delta-1$ neighbors in $G$ other than $u$, all of which are neighbors of $u$ in $G^2$. Also, $u$ has at most $\Delta-d$ other neighbors $v'$ in $G$, which come after $u$ in the ordering. Each such $v'$ has at most $d$ neighbors in $G$ preceding it, including $u$, so $v'$ has at most $d-1$ neighbors in $G$ that precede $u$.  Thus, $u$ has at most $d + d(\Delta-1)+(\Delta-d)(d-1)= \Delta(2d-1) + d - d^2$ neighbors in $G^2$ that precede it in the vertex ordering.
\end{proof}

\section{Bipartite $d$-degenerate graphs with one request}
\label{sec: alontarsi}

The theorem stated below follows directly from the celebrated Alon-Tarsi Theorem (see Theorem 1.1 in~\cite{alon1992colorings}).

\begin{thm}[\cite{alon1992colorings}]
\label{thm: alontarsi_original}
Let $D$ be a bipartite directed graph. For each $v \in V(D)$, let $L(v)$ be a list of $d^+_D(v) + 1$ distinct integers, where  $d^+_D(v)$ is the outdegree of $v$. Then, there is a proper vertex coloring of $D$ where every vertex $v$ gets a color from its list $L(v)$. 
\end{thm}
We use Theorem~\ref{thm: alontarsi_original} to prove Theorem~\ref{alontarsi}. 

\begin{proof}
Let $G$ be any $d$-degenerate bipartite graph with a $(d+1)$-list assignment $L$.  Suppose $r$ is a request of $L$ whose domain is of size $1$. Since $G$ is $d$-degenerate, its vertices can be ordered $v_1,\ldots, v_n$ such that for all $i\in[n]$, $v_i$ has at most $d$ neighbors $v_j$ with $j>i$.  Suppose that $v_k$ is the vertex in the domain of $r$.  We will show that there is a way to orient the edges of $G$ to obtain a digraph in which every vertex has out-degree at most $d$ and $v_k$ has out-degree zero. The desired result will then follow directly from Theorem~\ref{thm: alontarsi_original}. 

Fix an orientation of $G$ by orienting each edge $v_iv_j$ from $v_i$ to $v_j$ whenever $i<j$.  Then $d^+(v_i):= |\{v_j\in N(v_i):i<j\le n\}| \le d$.  Let $H$ be a (multi)digraph constructed from $G[\{v_i:k\le i\le n\}]$ by adding a new vertex $v_{n+1}$ to $G[\{v_i:k\le i\le n\}]$ and adding $d-d^+(v_i)$ edges from $v_i$ to $v_{n+1}$ for each $k\le i\le n$.  Orient each edge of $H$ from $v_i$ to $v_j$ where $i<j$.  Note that in $H$, every vertex except $v_{n+1}$ has out-degree equal to $d$.

We wish to apply the edge version of Menger's Theorem (Theorem~4.2.19 in~\cite{W01}) to show that there are $d$ edge-disjoint directed paths from $v_k$ to $v_{n+1}$ in $H$. 
For a contradiction, suppose that there is a set $S$ of fewer than $d$ edges in $H$ such that $H-S$ has no $v_k,v_{n+1}$-path.  Let $j$ be maximum such that there is a directed $v_k$,$v_j$-path in $H-S$.  We know $H$ has $d$ edges from $v_j$ to a vertex in $\{v_{j+1}, \ldots, v_{n+1}\}$.  These edges are clearly not all in $S$ which is a contradiction.  Thus, Menger's Theorem applies.

Fix a set of $d$ edge-disjoint directed $v_k,v_{n+1}$-paths in $H$ and let $F$ be the set of edges that appear in those paths.  If we reverse the orientation of each edge in $F$, the out-degree of each vertex in $H$ is unchanged except for $v_k$ and $v_{n+1}$.  Hence every vertex of $H$ still has out-degree at most $d$, and this remains true if $v_{n+1}$ is removed.

Now consider what happens to $G$ with these reversed edge orientations (of all edges of $F$ except for those incident to $v_{n+1}$):
$v_k$ will have out-degree zero, vertices $v_i$ with $i<k$ have out-degree unchanged, and vertices $v_i$ with $k<i\le n$ have the same out-degree they have in $H-v_{n+1}$ which is at most~$d$.  This is the orientation we sought.
\end{proof}

\section{List flexibility number versus related concepts} \label{versus}

We begin this section by showing that there are graphs $G$ for which $\chi_{flex}(G) > {4\over 3}\chi_{\ell}(G)$. In particular, Proposition~\ref{pro: oddrequest} below implies that for each $l \geq 100000$, $\chi_{flex}(K_{2l+1,t}) \geq {4\over 3}\chi_{\ell}(K_{2l+1,t})+{4\over 3}$ where $t = \sum_{i=0}^l \binom{2l+1}{2l+1-i} (2l)^i$.

\begin{pro} \label{pro: oddrequest}
Suppose that $l \in \N$ and $t_0 = \sum_{i=0}^l \binom{2l+1}{2l+1-i} (2l)^i$.\\
(i) For $t \ge t_0$, $\chi_{flex}(K_{2l+1,t}) \geq 2l+2$.\\
(ii) For $s \le t_0$, $\chi_{\ell}(K_{2l+1,s}) \le \lceil{3l/2}\rceil$ whenever $l \ge 100000$.
\end{pro}

\begin{proof}
We first prove Statement~(i).  Suppose that $G= K_{2l+1,t}$ with bipartition $X = \{x_1, \ldots, x_{2l+1} \}$, $Y = \{y_1, \ldots, y_t \}$.  Since $G$ is bipartite, $\rho(G)=2$.  We will construct a $(2l+1)$-assignment $L$ for $G$ and a request $r$ with domain $X$ with the property that $(G,L,r)$ is not $(1/2)$-satisfiable.  It immediately follows that $\chi_{flex}(G) \geq 2l+2$.

Let $A_1, \ldots, A_{2l+1}$ be any pairwise-disjoint $2l$-element subsets of $\{i\in \N: i>2l+1\}$. 
Let $L(x_i) = \{i\} \cup A_{i}$ for each $i \in [2l+1]$. 
Let $\mathcal C$ be the set of $(2l+1)$-element sets $C$ with exactly one element from $L(x_i)$ for each $i\in [2l+1]$, such that $C$ has at least $l+1$ elements in $[2l+1]$. 
Since each $C\in {\mathcal C}$ has $j$ elements in $[2l+1]$ for some $l+1\le j\le 2l+1$ and 
the number of sets $C\in {\mathcal C}$ with exactly $j$ elements in $[2l+1]$ is ${2l+1\choose j}(2l)^{2l+1-j}$, we see that $|{\mathcal C}| = t_0$ (where $i=2l+1-j$).  Index the elements of $\mathcal{C}$ so that $\mathcal{C} = \{C_1, \ldots, C_{t_0} \}$.  Finally, for each $j \in [t_0]$ let $L(y_j) = C_j$, and whenever $t_0 < j \leq t$ let $L(y_j) = [2l+1]$.


Let $r$ be the request of $L$ with domain $X$
for which $r(x_i) = i$ for each $i \in [2l+1]$.  For the sake of contradiction, suppose that $f$ is a proper $L$-coloring such that $|\{x \in X : f(x) = r(x) \}| \geq (1/2)(2l+1)$.  
Let $R = \{x \in X : f(x) = r(x) \}$ and $S = \{f(x): x \in R \}$. Then $|R|\ge \lceil{(1/2)(2l+1)}\rceil = l+1$ and, since the lists $L(x)$ for $x\in X$ are disjoint, $|S|=|R|$.  Hence $\{f(x) : x \in X \} = L(y)$ for some $y\in Y$. Thus, $f$ assigns $y$ a color that it also assigns to an element in $X$ which contradicts the fact that $f$ is proper.  

We now turn our attention to Statement~(ii).  Let $G= K_{2l+1,s}$ where $l \geq 100000$ and $s \le t_0$. Suppose $G$ has the bipartition $X,Y$ with $|X|=2l+1$ and $|Y|=s$. Given an arbitrary $k$-assignment $L$, where $k=\lceil{3l/2}\rceil$, we show that $G$ is $L$-colorable. 

Let $p=\ln l/l$. We partition the palette $\mathcal{L}= \bigcup_{v \in V(G)}L(v)$ into two sets $\mathcal{L}_X$ and $\mathcal{L}_Y$ by randomly 
assigning each color of $\mathcal{L}$ to $\mathcal{L}_X$ with probability $p$ and to $\mathcal{L}_Y$ with probability $1-p$.
If we can assign each vertex $x \in X$ a color from $L(x)\cap \mathcal{L}_X$ and each vertex $y \in Y$ a color from $L(y)\cap \mathcal{L}_Y$, we will have obtained a proper $L$-coloring of $G$, and that is doable if none of those sets are empty.  The expected number of sets $L(x)\cap \mathcal{L}_X$ and $L(y)\cap \mathcal{L}_Y$ with $x\in X$ and $y\in Y$ that are empty is $(2l+1)(1-p)^k+sp^k$ and $s \le (2l)^l 2^{2l+1}$, so it is enough to show that $(2l+1)(1-p)^k + (2l)^l 2^{2l+1} p^k <1$. It can be computationally verified that this inequality holds for $l\ge 100000$. We did not attempt to optimize this bound on $l$. 
\end{proof}

One might question whether maximizing $\epsilon$ is meaningfully different from a flexible list coloring with smaller $\epsilon>0$.  Are there any graphs $G$ and $k \in \N$ such that $G$ is $(k,\epsilon)$-flexible for some $\epsilon>0$, but $G$ is not $(k,1/\rho(G))$-flexible? We now work towards proving Proposition~\ref{pro: K37} which shows that the answer to this question is yes.  We begin with a lemma.

\begin{lem} \label{lem: K37}
Suppose that $G = K_{3,7}$.  Suppose that $L$ is a 3-assignment for $G$ and $r: D \rightarrow \bigcup_{v \in V(G)} L(v)$ is a request of $L$ with $|D|=1$.  Then, $(G,L,r)$ is $1$-satisfiable.
\end{lem}

\begin{proof}
Suppose that the bipartition of $G = K_{3,7}$ is $X = \{x_1, x_2, x_3\}$, $Y= \{y_1, \ldots, y_7 \}$.  Also, suppose that $D = \{z\}$ and $r(z) = c$.  We will prove a proper $L$-coloring that colors $z$ with $c$ exists in the case $z \in X$ and the case $z \in Y$.

First, suppose that $z \in X$, and assume without loss of generality that $z=x_1$.  Color $z$ with $c$, and let $L'(y_i) = L(y_i) - \{c\}$ for each $i \in [7]$.  Notice that if there is a $c' \in L(x_2) \cap L(x_3)$, we can complete a proper $L$-coloring by coloring $x_2$ and $x_3$ with $c'$ and then greedily coloring the vertices in $Y$.  So, suppose that $L(x_2) \cap L(x_3) = \emptyset$. Then, since $\{ \{a,b\} : a \in L(x_2), b \in L(x_3) \}$ has 9 elements, there is a $c' \in L(x_2)$ and a $c'' \in L(x_3)$ such that $\{c', c'' \} \neq L'(y_i)$ for each $i \in [7]$.  Let $L''(y_i) = L'(y_i) - \{c',c''\}$ for each $i \in [7]$, and note $|L''(y_i)| \geq 1$.  We can now complete a proper $L$-coloring by coloring $x_2$ with $c'$, $x_3$ with $c''$, and $y_i$ with an element of $L''(y_i)$ for each $i \in [7]$.

Next, suppose that $z \in Y$, and assume without loss of generality that $z=y_1$.  Color $z$ with $c$, and let $L'(x_i) = L(x_i) - \{c\}$ for each $i \in [3]$.  Notice that if there are $i,j \in [3]$ with $i \neq j$ and $L'(x_i) \cap L'(x_j) \neq \emptyset$, we can complete a proper $L$-coloring by coloring $x_i$ and $x_j$ with an element of $L'(x_i) \cap L'(x_j)$ and then greedily coloring what remains.  So, suppose $L'(x_1)$ , $L'(x_2)$, and $L'(x_3)$ are pairwise disjoint. Then, since $\{ \{a,b,d\} : a \in L'(x_1), b \in L'(x_2), d \in L'(x_3) \}$ has at least 8 elements, there is a $c' \in L(x_1)$, $c'' \in L(x_2)$, and $c''' \in L(x_3)$ such that $\{c', c'', c''' \} \neq L'(y_i)$ for each $i \in \{2, \ldots, 7 \}$.  Let $L'(y_i) = L(y_i) - \{c',c'',c'''\}$ for each $i \in \{2, \ldots, 7\}$, and note $|L'(y_i)| \geq 1$.  We can now complete a proper $L$-coloring by coloring $x_1$ with $c'$, $x_2$ with $c''$, $x_3$ with $c'''$, and $y_i$ with an element of $L'(y_i)$ for each $i \in \{2, \ldots, 7\}$.
\end{proof}

We can now prove Proposition~\ref{pro: K37}, that $G=K_{3,7}$ is $(3, 1/10)$-flexible and $\chi_{flex}(G) > 3$.

\begin{proof}
Recall that $\chi_{flex}(G) > 3$ immediately follows from Proposition~\ref{pro: oddrequest}. Suppose that $L$ is an arbitrary 3-assignment for $G$, and $r: D \rightarrow \bigcup_{v \in V(G)} L(v)$ is a request of $L$.  By Lemma~\ref{lem: K37} there is a proper $L$-coloring of $G$ that colors at least one vertex in $D$ according to $r$.  Since $|D| \leq 10$, we may conclude that $G$ is $(3,1/10)$-flexible.
\end{proof}

\section{The Cartesian product of graphs} \label{products}

Several researchers have considered list coloring (and generalizations of list coloring) of the Cartesian products of graphs (see~\cite{BJ06, KM17, KM18, KM182, KM21, M22}).  In this section we consider flexible list colorings of the Cartesian product of graphs.  

\begin{pro} \label{pro: Cartesian}
Suppose $G$ is $(k,\epsilon)$-flexible. Then  $G \square H$ is $(\max\{k, \Delta(H)+\chi_\ell(G)\}, \epsilon/\chi(H))$-flexible.
\end{pro}

\begin{proof}
Suppose $L$ is an arbitrary $\max\{k, \Delta(H)+\chi_\ell(G)\}$-assignment for $G \square H$, and suppose $r$ is a request of $L$ with domain $D$. Partition $V(H)$ by using the color classes of a proper $\chi(H)$-coloring.  There must be a color class such that the corresponding copies of $G$ in $G \square H$ contain at least $|D|/\chi(H)$ of the requests. Since $G$ is $(k,\epsilon)$-flexible, we can $L$-color each of those copies of $G$ so that at least $\epsilon$ proportion of those requests are satisfied.  Afterward, each uncolored vertex in $G \square H$ is adjacent to at most $\Delta(H)$ colored vertices.  If we delete those colors from its list, at least $\chi_\ell(G)$ colors remain, so we can $L$-color each remaining copy of $G$, which completes a proper $L$-coloring of $G \square H$.
\end{proof}

Next, we prove Proposition~\ref{pro: grid}: the grid $P_n \square P_m$ is $(3, 1/3)$-flexible.

\begin{proof}
Let $G=P_n \square P_m$ with a 3-assignment $L$.  Let $V(G)=\{ v_{ij}\,:\,1\le i\le n, 1\le j\le m\}$. 
We will construct a set $\mathcal C$ of 
$3\cdot 2^{m-1}$ $L$-colorings such that for all $c\in L(v_{ij})$, $1\le i\le n$, and $1\le j\le m$, the color $c$ appears in exactly $1/3$ of the colorings in $\mathcal C$.  
Applying Proposition~\ref{pro: generalized pack} will finish the proof.



We construct the colorings by induction on $m$.  For $m=1$, get the first $L$-coloring greedily, then remove those colors from the lists and use a list packing of size~2 to get the other $L$-colorings (recall from Section~\ref{subsec: list packing} that $\chi_\ell^*(P_n)\le 2$.).  Then each color on each vertex's list appears on exactly one of the three colorings, as required.  

Suppose we have a set of $3\cdot 2^{m-2}$ colorings $\mathcal B$ of $P_n \square P_{m-1}$ as required.  
For each $i\in[n]$, fix a bijection $f_i$ from $L(v_{i(m-1)})$ to $L(v_{im})$ such that $f_{i}(c)=c$ for each $c$ that is in both $L(v_{i(m-1)})$ and $L(v_{im})$.
We will extend each coloring $g\in {\mathcal B}$ to two colorings of $G$, obtaining a set of $2|{\mathcal B}|=3\cdot 2^{m-1}$ colorings of $G$ which we will call $\mathcal C$.

For each $g\in {\mathcal B}$ and $i\in [n]$, let $L'(v_{im}) = L(v_{im})-f_{i}(g(v_{i(m-1)}))$ and then use an $L'$-packing of size~2 to get two $L'$-colorings.
We need to check that for all $i\in[n]$, $j\in[m]$, and $c\in L(v_{ij})$, $c$ appears in exactly 1/3 of the colorings of $\mathcal C$.  For $j<m$, $c$ is in 1/3 of the colorings of $\mathcal B$ and each such coloring has been extended in two ways to create $\mathcal C$, so it's true of $c$ as well.  For $j=m$: $1/3$ of the colorings $g\in{\mathcal B}$ have $f_i(g(v_{i(m-1)})=c$ and their two extensions to $G$ in $\mathcal C$ do not use $c$ on $v_{im}$.  Two-thirds of the colorings of $g\in{\mathcal B}$ use the other two colors in $L(v_{i(m-1)})$, and each of those colorings extends in two ways, one of which uses $c$ on $v_{im}$.  Since half the extensions of two-thirds of the colorings of $\mathcal C$ use $c$ on $v_{im}$, it means that $1/3$ of the colorings of $\mathcal C$ use $c$ on $v_{im}$, as desired.
\end{proof}

Next, we prove Proposition~\ref{pro: ladder}, that the $n$-ladder is $(3,1/2)$-flexible, which implies that $\chi_{flex}(P_2 \square P_n)=3$.

\begin{proof}
Let $G_n = P_2  \square P_n$ with $V(G_n)=\{ v_{ij}\,:\,1\le i\le 2, 1\le j\le n\}$.  Let $G_n' = G_n-v_{2n}$ and let $\mathcal G$ be the collection of all $G_n$ and $G_n'$ for $n\ge 1$.  We will show by induction on the number of vertices that each graph in $\mathcal G$ with any 3-assignment $L$ is $(3,1/2)$-flexible.

Let $G\in {\mathcal G}$ and let $L$ be a 3-assignment and let $r$ be an arbitrary request of $L$ with domain $D\subseteq V(G)$; so, $r(v)\in L(v)$ for all $v\in D$.  We must find an $L$-coloring of $G$ such that at least $|D|/2$ of the 
vertices $v\in D$ get color $r(v)$.

$G_1'$, $G_1$, and $G_2'$ are paths, for which this is known.  So we may assume that $G=G_n$ or $G_n'$ with $n\ge 2$.

If $G=G_n$ and $v_{2n}$ is not in $D$, delete $v_{2n}$ and apply induction to properly $L$-color the remaining graph $G_n'$; then at least half the vertices $v\in D$ will get the requested color $r(v)$.
Since $v_{2n}$ has only 2~neighbors in $G_n'$, we can extend the $L$-coloring to $v_{2n}$.   
Thus, we may assume that $v_{2n}$ and $v_{1n}$ (by symmetry) are in $D$ for $G=G_n$.  More generally, given a proper $L$-coloring of an induced subgraph $H$, we can always extend to a proper $L$-coloring of the full graph $G$ by greedy coloring if we can order the remaining vertices so that each has at most~2 neighbors that are in $H$ or (not in $H$ and) earlier in the ordering.  However, this will not necessarily fill any of the requests on $V(G)\setminus V(H)$.

If $G=G_n$ and $r(v_{1n})=r(v_{2n})$, apply induction to properly $L$-color $G_{n-1}$.  Since $v_{1(n-1)}$ and $v_{2(n-1)}$ are adjacent, at least one of the $v_{i(n-1)}$ for $i=1,2$ does not have color $r(v_{in})$. Hence, we can give that $v_{in}$ color $r(v_{in})$.  Then we can extend the proper $L$-coloring to $v_{(2-i)n}$ (since it has only~2 neighbors).  Since at least half of the requests on $G_{n-1}$ were fulfilled and we fulfilled at least~1 of the 2 requests on the other two vertices, this $L$-coloring of $G_n$ suffices.

For $G=G_n$, we may henceforth assume that $r(v_{1n})=1$ and $r(v_{2n})=2$.  Apply induction to get a proper $L$-coloring $f$ of $G_{n-2}$.  Suppose that $f(v_{1(n-2)})=a$ and $f(v_{2(n-2)})=b$.  We can color both $v_{1n}$ and $v_{2n}$ with colors $1$ and $2$ as requested and extend the $L$-coloring to $v_{1(n-1)}$ and $v_{2(n-1)}$ unless $L(v_{1(n-1)})=\{1,a,c\}$ and $L(v_{2(n-1)})=\{2,b,c\}$ for some color $c$.  Since at least half the requests on $V(G_{n-2})$ were satisfied and two (or more) of the other four vertices would get their requested color, this proper $L$-coloring would suffice.  Hence, without loss of generality we may assume that $L(v_{1(n-1)})=\{1,a,c\}$,  $L(v_{2(n-1)})=\{2,b,c\}$, where $a\in L(v_{1n})$ and $b\in L(v_{2n})$.

If neither $v_{1(n-1)}$ nor $v_{2(n-1)}$ are in $D$, we color $G_{n-2}$ by $f$, give $v_{1n}$ color~1 as requested, and then extend the proper $L$-coloring by greedily $L$-coloring $v_{1(n-1)}$, $v_{2(n-1)}$, $v_{2n}$, in that order.  This suffices since there our coloring fulfilled one of the two requests among tho four vertices not in $G_{n-2}$. Hence, without loss of generality we may assume that $v_{1(n-1)}\in D$.

If $r(v_{1(n-1)})=1$, then we can fulfil the requests at $v_{1(n-1)}$ and $v_{2n}$ (with colors 1,2 respectively), color $G_{n-2}$ by $f$, give $v_{2(n-1)}$ color $c$, finally greedy $L$-color $v_{1n}$. Note that we have a proper $L$-coloring and at least two requests are satisfied on the four vertices not in $G_{n-2}$, so this suffices.

If $r(v_{1(n-1)})\not=1$ (meaning it is either $a$ or $c$), properly $L$-color $G_{n-1}$ by induction. Note that we only need to satisfy one additional request.  If $v_{1(n-1)}$ does not have color~1, then color $v_{1n}$ with color~1 (satisfying its request), then greedily extend the proper $L$-coloring to $v_{2n}$; this suffices.  If $v_{1(n-1)}$ has color~1, then we will replace the colors on $v_{1(n-1)}$ and $v_{2(n-1)}$ by recoloring $v_{1(n-1)}$ with its requested color $r(v_{1(n-1)})$, then greedily $L$-color $v_{2(n-1)}$; while it is possible that $v_{2(n-1)}$ had and no longer has its requested color, the overall number of requests satisfied has not decreased.  The modified $L$-coloring of $G_{n-1}$ extends to $G$ as before: give $v_{1n}$ the color~1 (as requested), then greedily $L$-color $v_{2n}$.  

Now it remains to consider the case $G=G_n'$ for $n\ge 2$.  If $v_{1n}\not\in D$, apply induction to properly $L$-color the subgraph $G_{n-1}$, then greedy $L$-color $v_{1n}$.  That suffices, so henceforth we can assume that $v_{1n}\in D$ and also $r(v_{1n})=1$.

We could apply induction to properly $L$-color $G_{n-2}$, color $v_{1n}$ with color~1, then greedily $L$-color $v_{1(n-1)}$ and then $v_{2(n-1)}$.  This $L$-coloring satisfies at least half the requests on $G_{n-2}$ and one more at $v_{1n}$, which suffices $v_{1(n-1)}$ and $v_{2(n-1)}$ are also in $D$.
Thus, we can assume that $D$ includes $v_{1n}$, $v_{1(n-1)}$, and $v_{2(n-1)}$.

If $r(v_{1(n-1)}) = r(v_{1n}) = 1$, remove those vertices and apply induction to properly $L$-color the remaining graph (since it's isomorphic to $G_{n-1}'$).  We can greedily extend the proper $L$-coloring to $v_{1(n-1)}$, then to $v_{1n}$.  Since $v_{1(n-1)}$ is the only neighbor of $v_{1n}$, if $v_{1(n-1)}$ does not have color~1, then we can give color~1 to $v_{1n}$. Thus at least one of the requests on these two vertices is satisfied, which suffices.  Thus, we may assume that $r(v_{1(n-1)})\not=1$.

Apply induction to $L$-color $G_{n-1}'$.  If $v_{1(n-1)}$ did not get color~1, we can color $v_{1n}$ with~1 (its requested color) then greedily $L$-color $v_{2(n-2)}$; this suffices since we colored two additional vertices and satisfied at least one request.  If $v_{1(n-1)}$ has color~1, then its request is not currently satisfied, so recoloring it will not decrease the number of requests fulfilled.  Since $v_{1(n-1)}$ has degree~1 in $G_{n-1}'$, there are at least two colors in $L(v_{1(n-1)})$ which are not on its neighbor in $G_{n-1}'$; recolor $v_{1(n-1)}$ with the one that isn't~1, then proceed as before.
\end{proof}

\section{The join of two copies of a graph}\label{joins}

In this section we consider the list flexibility number of the join of two copies of the same graph.  Our motivation for studying this was to better understand the relationship between the list flexibility number and the degeneracy of a graph. Can we do better than the trivial degeneracy bound on the list flexibility number?  More specifically, we wanted to use the join operation to make some initial progress on Question~\ref{ques: lowerbounddegen}.  Suppose that $G_1$ and $G_2$ are vertex disjoint graphs, each isomorphic to some graph $G$.  Let $H = G_1 \vee G_2$.  It is easy to see that $\rho(H) = 2\rho(G)$.

For the theorem below, let $h: [0,1] \rightarrow \mathbb{R}$ be the (binary) entropy function given by $h(p) = -p \log_2(p) - (1-p) \log_2(1-p)$.

\begin{thm} \label{thm: generaljoin}
Suppose $G$ is an $n$-vertex, $s$-choosable graph with $\chi_{flex}(G) = m$ and $n \geq 2$.  Let $l = \lceil n/\rho(G) \rceil$.  Suppose $G_1$ and $G_2$ are vertex disjoint graphs, both of which are isomorphic to $G$.  Let $H = G_1 \vee G_2$.  Then, for any real number $r > 2$,
$$\chi_{flex}(H) \leq \max \left \{ \left \lceil l + \frac{\log_2(2n-l)}{1-h(1/r)} \right \rceil, \lceil r(s-1) + l \rceil, m \right \}.$$
\end{thm}

\begin{proof}
Suppose $k = \max \left \{ \left \lceil l + \frac{\log_2(2n-l)}{1-h(1/r)} \right \rceil, \lceil r(s-1) + l \rceil, m \right \}$, and $L$ is an arbitrary $k$-assignment for $H$.  Furthermore, suppose that $r: D \rightarrow \bigcup_{v \in D} L(v)$ is an arbitrary request of $L$.  In order to prove the desired, we will show that $(H,L,r)$ is $(1/(2\rho(G)))$-satisfiable.  Specifically, we will show that there is a proper $L$-coloring of $H$ that colors at least $|D|/(2\rho(G))$ of the vertices in $D$ according to $r$. 

Suppose without loss of generality that $|D \cap V(G_1)| \geq |D \cap V(G_2)|$.  Let $L_1$ be the $k$-assignment for $G_1$ obtained by restricting the domain of $L$ to $V(G_1)$, and let $r_1$ be the request of $L_1$ obtained by restricting the domain of $r$ to $D \cap V(G_1)$.  Since $k \geq m$, there is proper $L_1$-coloring of $G_1$, $f$, such that at least $|D \cap V(G_1)|/\rho(G)$ of the vertices in $D \cap V(G_1)$ are colored according to $r_1$.  

Suppose we color each vertex in $V(G_1)$ according to $f$.  Now, arbitrarily uncolor vertices in $V(G_1)$ so that exactly $n-l$ are no longer colored and of the $l$ colored vertices, at least $|D \cap V(G_1)|/\rho(G)$ of the vertices in $D \cap V(G_1)$ are colored according to $r_1$ (this is possible since $|D \cap V(G_1)|/\rho(G) \leq l$).  Call the set of $l$ vertices that remain colored $R$.  Notice that at least $|D \cap V(G_1)|/\rho(G)$ of the vertices in $D$ have been colored according to $r$, and
$$\frac{|D \cap V(G_1)|}{\rho(G)} = \frac{2|D \cap V(G_1)|}{2\rho(G)} \geq \frac{|D|}{2 \rho(G)}.$$
Let $H' = H - R$.  For each $v \in V(H')$, let $L'(v) = L(v) - \{f(u) : u \in N_H(v) \cap R \}$.  Clearly, $|L'(v)| \geq k-l$ for each $v \in V(H')$.  In order to complete the proof, we must show that there is a proper $L'$-coloring of $H'$.

Our proof that a proper $L'$-coloring of $H'$ exists is probabilistic.  For each $c \in \bigcup_{v \in V(H')} L'(v)$, independently and uniformly at random place $c$ in one of two sets: $B_1$ and $B_2$.  Since $G$ is $s$-choosable, notice that a proper $L'$-coloring of $H'$ will exist if $|L'(v) \cap B_1| \geq s$ for each $v \in V(G_1)-R$ and $|L'(v) \cap B_2| \geq s$ for each $v \in V(G_2)$.

For each $v \in V(H')$, let $E_v$ be the event that $|L'(v) \cap B_1| \geq s$ if $v \in V(G_1)-R$ and the event that $|L'(v) \cap B_2| \geq s$ if $v \in V(G_2)$.  Note that a proper $L'$-coloring of $H'$ exists if $P[ \bigcap_{v \in V(H')} E_v] > 0$.  Since $|L'(v)| \geq k-l$ for each $v \in V(H')$,
$$\mathbb{P}[E_v] \geq 1 - \sum_{i=0}^{s-1} \binom{k-l}{i} (1/2)^{k-l}.$$
By this bound, the fact that the events in $\{ \overline{E_v} : v \in V(H') \}$ are not pairwise disjoint, and the union bound,
$$\mathbb{P}\left[ \bigcap_{v \in V(H')} E_v\right] = 1 - \mathbb{P}\left[ \bigcup_{v \in V(H')} \overline{E_v}\right] > 1 - \sum_{v \in V(H')} \mathbb{P}[\overline{E_v}] \geq 1 - (2n-l)\sum_{i=0}^{s-1} \binom{k-l}{i} (1/2)^{k-l}.$$
Now, since $k \geq r(s-1) + l$, we know that $1/r \geq (s-1)/(k-l)$ which means by the fact that $r > 2$ and a well-known bound on the sum of binomial coefficients (see e.g.,~\cite{G14}), $$\sum_{i=0}^{s-1} \binom{k-l}{i} (1/2)^{k-l} \leq 2^{(k-l)(h(1/r)-1)}.$$
This bound along with the fact that $k \geq  l + \frac{\log_2(2n-l)}{1-h(1/r)}$ yields
$$ 1 - (2n-l)\sum_{i=0}^{s-1} \binom{k-l}{i} (1/2)^{k-l} \geq 1 - (2n-l)2^{(k-l)(h(1/r)-1)} \geq 0.$$
\end{proof}

For example, if $G = P_{100}$ and we let $r=12$, Theorem~\ref{thm: generaljoin} implies that $\chi_{flex}(P_{100} \vee P_{100}) \leq 63$.  Notice that 63 is quite a bit lower than the degeneracy of $P_{100} \vee P_{100}$ which is 101.  We will now specifically focus on $P_n \vee P_n$.  Our next proposition demonstrates that it is possible to improve upon the bound in Theorem~\ref{thm: generaljoin} when we know more about the structure of $G$.

\begin{pro} \label{pro: joinpath}
Suppose $n$ is a positive even integer, and $G = P_n \vee P_n$.  If $p \in (0,1)$ and $k > \max \{2, n/2 \}$ satisfy
$$(n/2)(1-p)^{k-2} + np^{k-1-n/2}(p + (k - n/2)(1-p)) \leq 1,$$
then $\chi_{flex}(G) \leq k$.
\end{pro}

\begin{proof}
Suppose $L$ is an arbitrary $k$-assignment for $G$.  Furthermore, suppose that $r: D \rightarrow \bigcup_{v \in D} L(v)$ is an arbitrary request of $L$.  In order to prove the desired, we will show that $(H,L,r)$ is $(1/4)$-satisfiable.

Suppose that the vertices (in order) of the first copy of $P_n$, which we will denote as $G_1$, used to form $G$ are $x_1, \ldots, x_n$, and suppose that the vertices (in order) of the second copy of $P_n$, which we will denote as $G_2$, used to form $G$ are $y_1, \ldots, y_n$.  Let $X_1, X_2$ be the unique bipartition of $G_1$, and let $Y_1, Y_2$ be the unique bipartition of $G_2$.  Note that $|X_1|=|X_2|=|Y_1|=|Y_2| = n/2$.  We will assume without loss of generality that $|X_1 \cap D| \geq |D|/4$.

Now, for each $v \in X_1 \cap D$ color $v$ with $r(v)$, and for each $v \in X_1 - D$ arbitrarily color $v$ with an element of $L(v)$.  Call this resulting partial $L$-coloring of $G$, $f$.

Let $G' = G - X_1$.  For each $v \in V(G')$, let $L'(v) = L(v) - \{f(u) : u \in N_G(v) \cap X_1 \}$.  Clearly, $|L'(v)| \geq k-2$ for each $v \in X_2$, and $|L'(v)| \geq k-n/2$ for each $v \in Y_1 \cup Y_2$.  In order to complete the proof, we must show that there is a proper $L'$-coloring of $G'$.

Our proof that a proper $L'$-coloring of $G'$ exists is probabilistic.  For each $c \in \bigcup_{v \in V(G')} L'(v)$, independently and with probability $p$ place $c$ into set $B_1$ and place $c$ into set $B_2$ otherwise.  Since $G'[X_2]$ is $1$-choosable and $G'[Y_1 \cup Y_2]$ is $2$-choosable, notice that a proper $L'$-coloring of $G'$ will exist if $|L'(v) \cap B_1| \geq 1$ for each $v \in X_2$ and $|L'(v) \cap B_2| \geq 2$ for each $v \in Y_1 \cup Y_2$.

For each $v \in V(G')$, let $E_v$ be the event that $|L'(v) \cap B_1| \geq 1$ if $v \in X_2$ and the event that $|L'(v) \cap B_2| \geq 2$ if $v \in Y_1 \cup Y_2$.  Note that a proper $L'$-coloring of $G'$ exists if $\mathbb{P}[ \bigcap_{v \in V(G')} E_v] > 0$.  Since $|L'(v)| \geq k-2$ for each $v \in X_2$,
$\mathbb{P}[E_v] \geq 1 - (1-p)^{k-2}$ when $v \in X_2$, and since $|L'(v)| \geq k-n/2$ for each $v \in Y_1 \cup Y_2$,
$\mathbb{P}[E_v] \geq 1 - p^{k-n/2} - (k-n/2)p^{k-n/2-1}(1-p)$ when $v \in Y_1 \cup Y_2$. 

By these bounds, the fact that the events in $\{ \overline{E_v} : v \in V(G') \}$ are not pairwise disjoint, and the union bound,
\begin{align*}
\mathbb{P}\left[ \bigcap_{v \in V(G')} E_v\right] &= 1 - \mathbb{P}\left[ \bigcup_{v \in V(G')} \overline{E_v}\right] \\
&> 1 - \sum_{v \in V(G')} \mathbb{P}[\overline{E_v}] \\
&\geq 1 - (n/2)(1-p)^{k-2} - np^{k-1-n/2}(p + (k - n/2)(1-p)) \\
&\geq 0.
\end{align*}
\end{proof}

We immediately get the following corollary.

\begin{cor} \label{cor: paths}
Suppose $n$ is a positive even integer satisfying $n \geq 50$.  Then, $\chi_{flex}(P_n \vee P_n) \leq \lceil n/2 + \ln(n) \rceil$.
\end{cor}

\begin{proof}
Notice that the hypotheses of Proposition~\ref{pro: joinpath} are satisfied when $n$ is a positive even integer satisfying $n \geq 50$, $p = 2 \ln(n)/n$, and $k = \lceil n/2 + \ln(n) \rceil$.  The result immediately follows.
\end{proof}

Note that Corollary~\ref{cor: paths} implies $\chi_{flex}(P_{100} \vee P_{100}) \leq 55$ which is a significant improvement on the result obtained from Theorem~\ref{thm: generaljoin}.  In the context of Question~\ref{ques: lowerbounddegen}, note that since the degeneracy of $P_n \vee P_n$ is $n+1$, Corollary~\ref{cor: paths} implies $\mathcal{F}(d)$ grows no faster than $d/2$ as $d \rightarrow \infty$. 

\medskip

{\bf Acknowledgment.}  The second author was supported by a grant from the Science and Engineering Research Board, Department of Science and Technology, Govt. of India (project number: MTR/2019/000550). The third author received support from the College of Lake County to conduct research related to this paper during the summer 2022 semester.  The support of the College of Lake County is gratefully acknowledged. Finally, we thank the anonymous referees for their careful reading and many suggestions that improved the paper.


\end{document}